\documentclass{amsart}

\usepackage{latexsym,enumerate}
\usepackage{amsmath,amsthm,amsopn,amstext,amscd,amsfonts,amssymb}
\usepackage[ansinew]{inputenc}
\usepackage{verbatim}
\usepackage{graphicx}
\usepackage{wasysym}
\usepackage[all]{xy}
\usepackage{pstricks}
\usepackage{epsfig} % Para gráficos
\usepackage{psfrag} % Mejor éste para gráficos: es más moderno
\usepackage{subfigure}

\ifx\volno\undefined\def\volno{16}\fi
\ifx\volyear\undefined\def\volyear{2009}\fi
\ifx\papno\undefined\def\papno{R00}\fi

\newtheorem{theorem}{Theorem}[section]
\newtheorem{lemma}[theorem]{Lemma}
\newtheorem{proposition}[theorem]{Proposition}
\newtheorem{corollary}[theorem]{Corollary}
\newtheorem{definition}[theorem]{Definition}
\newtheorem{example}[theorem]{Example}
\newtheorem{remark}[theorem]{Remark}

\DeclareMathOperator{\diam}{diam}

\newcommand{\spb}[1]{\smallskip}
\newcommand{\mpb}[1]{\medskip}
\newcommand{\bpb}[1]{\bigskip}

\renewcommand{\a}{\alpha}

\renewcommand{\d}{\delta}
\newcommand{\g}{\gamma}
\newcommand{\G}{\Gamma}

\newcommand{\s}{\sigma}
\begin{document}

\title{Gromov hyperbolicity in lexicographic product graphs}

\author[Walter Carballosa]{Walter Carballosa$^{(1)}$}
\address{National council of science and technology (CONACYT) $\&$ Autonomous University of Zacatecas,
Paseo la Bufa, int. Calzada Solidaridad, 98060 Zacatecas, ZAC, Mexico}
\email{waltercarb@gmail.com}
\thanks{$^{(1)}$ Supported in part by a grant from Ministerio de Econom{\'\i}a y Competitividad (MTM 2013-46374-P) Spain.}

\author[Amauris de la Cruz]{Amauris de la Cruz$^{(1)}$}
\address{Department of Mathematics, Universidad Carlos III de Madrid,
Avenida de la Universidad 30, 28911 Legan\'es, Madrid, Spain}
\email{alcruz@math.uc3m.es}

\author[Jos\'e M. Rodr{\'\i}guez]{Jos\'e M. Rodr{\'\i}guez$^{(1,2)}$}
\address{Department of Mathematics, Universidad Carlos III de Madrid,
Avenida de la Universidad 30, 28911 Legan\'es, Madrid, Spain}
\email{jomaro@math.uc3m.es}
\thanks{$^{(2)}$ Supported in part by a grant from CONACYT (CONACYT-UAG I0110/62/10), Mexico.}

%\centerline{ Walter Carballosa, Amauris de la Cruz}
%\smallskip
%\centerline{ Departamento de Matem\'aticas}
%\centerline{ Universidad Carlos III de Madrid}
%\centerline{Av. de la Universidad 30, 28911 Legan\'es, Madrid, Spain}
%\centerline{\tt wcarball@math.uc3m.es, alcruz@math.uc3m.es}
%\bigskip
%\centerline{ Jos\'e M. Rodr{\'\i}guez}
%\smallskip
%\centerline{ Departamento de Matem\'aticas}
%\centerline{ Universidad Carlos III de Madrid}
%\centerline{Av. de la Universidad 30, 28911 Legan\'es, Madrid, Spain}
%\centerline{\tt jomaro\@@math.uc3m.es}

\date{}

\maketitle{}

\begin{abstract}
If $X$ is a geodesic metric space and $x_1,x_2,x_3\in X$, a {\it
geodesic triangle} $T=\{x_1,x_2,x_3\}$ is the union of the three
geodesics $[x_1x_2]$, $[x_2x_3]$ and $[x_3x_1]$ in $X$. The space
$X$ is $\delta$-\emph{hyperbolic} $($in the Gromov sense$)$ if any side
of $T$ is contained in a $\delta$-neighborhood of the union of the two
other sides, for every geodesic triangle $T$ in $X$.
If $X$ is hyperbolic, we denote by
$\delta(X)$ the sharp hyperbolicity constant of $X$, i.e.
$\delta(X)=\inf\{\delta\ge 0: \, X \, \text{ is $\delta$-hyperbolic}\}.$
In this paper we characterize the lexicographic product of two graphs $G_1\circ G_2$ which are hyperbolic,
in terms of $G_1$ and $G_2$:
the lexicographic product graph $G_1\circ G_2$ is hyperbolic if and only if $G_1$ is hyperbolic, unless if $G_1$ is a trivial graph (the graph with a single vertex); if $G_1$ is trivial, then $G_1\circ G_2$ is hyperbolic if and only if $G_2$ is hyperbolic.
In particular, we obtain the sharp inequalities $\delta(G_1)\le \delta(G_1\circ G_2) \le \delta(G_1) + 3/2$ if $G_1$ is not a trivial graph, and we characterize the graphs for which the second inequality is attained.
\end{abstract}



%--------------------------- Introduction ----------------------------

\section{Introduction}

Hyperbolic spaces play an important role in geometric
group theory and in the geometry of negatively curved
spaces (see \cite{ABCD, GH, G1}).
The concept of Gromov hyperbolicity grasps the essence of negatively curved
spaces like the classical hyperbolic space, Riemannian manifolds of
negative sectional curvature bounded away from $0$, and of discrete spaces like trees
and the Cayley graphs of many finitely generated groups. It is remarkable
that a simple concept leads to such a rich
general theory (see \cite{ABCD, GH, G1}).

The different kinds of products of graphs are an important research topic in graph theory, applied mathematics and computer science.
Some large graphs are composed from some existing smaller ones by using several products of graphs,
and many properties of such large graphs are strongly associated with that of the corresponding smaller ones.
In particular, the lexicographic product of graphs has been extensively investigated in relation to a wide range of subjects
(see, \emph{e.g.}, \cite{KPT,Po,SSUA,YX,ZLM} and the references therein).

The first works on Gromov hyperbolic spaces deal with
finitely generated groups (see \cite{G1}).
Initially, Gromov spaces were applied to the study of automatic groups in the science of computation
(see, \emph{e.g.}, \cite{O}); indeed, hyperbolic groups are strongly geodesically automatic, \emph{i.e.}, there is an automatic structure on the group \cite{Cha}.

The concept of hyperbolicity appears also in discrete mathematics, algorithms
and networking. For example, it has been shown empirically
in \cite{ShTa} that the internet topology embeds with better accuracy
into a hyperbolic space than into an Euclidean space
of comparable dimension; the same holds for many complex networks, see \cite{KPKVB}.
A few algorithmic problems in
hyperbolic spaces and hyperbolic graphs have been considered
in recent papers (see \cite{ChEs, Epp, GaLy, Kra}).
Another important
application of these spaces is the study of the spread of viruses through on the
internet (see \cite{K21,K22}).
Furthermore, hyperbolic spaces are useful in secure transmission of information on the
network (see \cite{K27,K21,K22,NS}).

%The study of Gromov hyperbolic graphs is a subject of increasing interest; see, \emph{e.g.}, \cite{BRS,
%BRSV2,BRST,BPK,BHB1,CPRS,%CPoRS,
%CRS,CRS2,CRSV,CDEHV,K50,
%K27,K21,K22,K23,K24,K56,KPKVB,MRSV,MRSV2,NS,PeRSV,PRST,PRSV,
%%PRT1,
%PT,R,
%%RS,
%RSVV,
%%RT1,
%S,S2,T,WZ} and the references therein.

%If $(X,d)$ is a metric space and $\g:[a,b]\longrightarrow X$ is a continuous function,
%we define the \emph{length} of $\g$ as
%\[
%L(\g):=\sup\Big\{ \sum_{i=1}^n d(\g(t_{i-1}),\g(t_{i})):\,
%a=t_0<t_1<\cdots <t_n=b\Big\}\,.
%\]
If $X$ is a metric space we say that the curve $\g:[a,b]\longrightarrow X$ is a
\emph{geodesic} if we have $L(\g|_{[t,s]})=d(\g(t),\g(s))=|t-s|$ for every $s,t\in [a,b]$
(then $\gamma$ is equipped with an arc-length parametrization).
The metric space $X$ is said \emph{geodesic} if for every couple of points in
$X$ there exists a geodesic joining them; we denote by $[xy]$
any geodesic joining $x$ and $y$; this notation is ambiguous, since in general we do not have uniqueness of
geodesics, but it is very convenient.
Consequently, any geodesic metric space is connected.
If the metric space $X$ is
a graph, then the edge joining the vertices $u$ and $v$ will be denoted by $[u,v]$.

%Along the paper we just consider connected graphs with every edge of length $1$.
In order to consider a graph $G$ as a geodesic metric space, identify (by an isometry)
any edge $[u,v]\in E(G)$ with the interval $[0,1]$ in the real line;
then the edge $[u,v]$ (considered as a graph with just one edge)
is isometric to the interval $[0,1]$.
Thus, the points in $G$ are the vertices and, also, the points in the interior
of any edge of $G$.
In this way, any graph $G$ has a natural distance defined on its points, induced by taking shortest paths in $G$,
and we can see $G$ as a metric graph.
Throughout this paper, $G=(V,E)$ denotes a simple connected graph such that every edge has length $1$.
These properties guarantee that any graph is a geodesic metric space.
Note that to exclude multiple edges and loops is not an important loss of generality, since \cite[Theorems 8 and 10]{BRSV2} reduce the problem of compute
the hyperbolicity constant of graphs with multiple edges and/or loops to the study of simple graphs.

Consider a polygon $J=\{J_1,J_2,\dots,J_n\}$
with sides $J_j\subseteq X$ in a geodesic metric space $X$.
We say that $J$ is $\d$-{\it thin} if for
every $x\in J_i$ we have that $d(x,\cup_{j\neq i}J_{j})\le \d$.
Let us
denote by $\d(J)$ the sharp thin constant of $J$, \emph{i.e.},
$\d(J):=\inf\{\d\ge 0: \, J \, \text{ is $\d$-thin}\,\}\,. $
If $x_1,x_2,x_3$ are three points in $X$, a {\it geodesic triangle} $T=\{x_1,x_2,x_3\}$ is
the union of the three geodesics $[x_1x_2]$, $[x_2x_3]$ and
$[x_3x_1]$ in $X$.
We say that $X$ is $\d$-\emph{hyperbolic} if every geodesic
triangle in $X$ is $\d$-thin, and we denote by $\d(X)$ the sharp
hyperbolicity constant of $X$, \emph{i.e.}, $\d(X):=\sup\{\d(T): \, T \,
\text{ is a geodesic triangle in }\,X\,\}.$ We say that $X$ is
\emph{hyperbolic} if $X$ is $\d$-hyperbolic for some $\d \ge 0$; then $X$ is hyperbolic if and only if
$ \d(X)<\infty$. A geodesic \emph{bigon} is a geodesic triangle $\{x_1,x_2,x_3\}$ with $x_2=x_3$. Therefore, every bigon in a $\d$-hyperbolic geodesic metric space is $\d$-thin.

Trivially, any bounded
metric space $X$ is $(\diam X)$-hyperbolic.
A normed linear space is hyperbolic if and only if it has dimension one.
A geodesic space is $0$-hyperbolic if and only if it is a metric tree.
If a complete Riemannian manifold is simply connected and their sectional curvatures satisfy
$K\leq c$ for some negative constant $c$, then it is hyperbolic.
See the classical references \cite{ABCD,GH} in order to find more background and further results.

We want to remark that the main examples of hyperbolic graphs are the trees.
In fact, the hyperbolicity constant of a geodesic metric space can be viewed as a measure of
how ``tree-like'' the space is, since those spaces $X$ with $\delta(X) = 0$ are precisely the metric trees.
This is an interesting subject since, in
many applications, one finds that the borderline between tractable and intractable
cases may be the tree-like degree of the structure to be dealt with
(see, \emph{e.g.}, \cite{CYY}).

A main problem in the theory is to characterize in a simple way the hyperbolic graphs. Given a Cayley graph (of a presentation with solvable word problem)
there is an algorithm which allows to decide if it is hyperbolic.
However, for a general graph deciding whether or not a space is hyperbolic is a very difficult problem.
Therefore, it is interesting to study the hyperbolicity of particular classes of graphs.
The papers \cite{BRST,BHB1,CCCR,CRS2,CRSV,MRSV2,PeRSV,PRSV,R,Si,WZ} study the hyperbolicity of, respectively, complement of graphs, chordal graphs, strong product graphs, corona and join of graphs,
line graphs, Cartesian product graphs, cubic graphs, tessellation graphs, short graphs, median graphs and  $k$-chordal graphs.
In \cite{CCCR,CRS2,MRSV2} the authors characterize the hyperbolic product graphs (for strong product, corona and join of graphs, and Cartesian product) in terms of properties of their factor graphs.

The study of lexicographic product graphs is a subject of increasing interest (see, \emph{e.g.}, \cite{KPT,Po,SSUA,YX,ZLM} and the references therein).
In this paper we characterize the hyperbolic lexicographic product of two graphs $G_1\circ G_2$, in terms of $G_1$ and $G_2$:
if $G_1$ has at least two vertices, then $G_1\circ G_2$ is hyperbolic if and only if $G_1$ is hyperbolic; besides, if $G_1$ has a single vertex, then $G_1\circ G_2$ is hyperbolic if and only if $G_2$ is hyperbolic (see Theorem \ref{t:Caracterizacion 1} and Remark \ref{r:d_Trivial}).
We also prove the sharp inequalities $\delta(G_1)\le \delta(G_1\circ G_2) \le \delta(G_1) + 3/2$ if $G_1$ is not a trivial graph, see Theorems \ref{t:sublex} and \ref{t:CotaSup};
Example \ref{ex:Cn_P2} provides a family of graphs for which the first inequality is attained;
besides, Theorems \ref{th:equal+3/2} and \ref{th:hyp3/2_diam2} characterize the graphs for which the second inequality is attained.

Furthermore, we obtain the precise value of the hyperbolicity constant for many lexicographic products (see Examples \ref{ex:Pn_P2}, \ref{ex:Cn_P2} and Theorem \ref{t:tree_graph}).
In particular, Theorem \ref{t:tree_graph} allows to compute, in a simple way, the hyperbolicity constant of the lexicographic product of any tree and any graph.

%SECTION  2------------------------------------------------------------------

\
\section{Distances in lexicographic products}

In order to estimate the hyperbolicity constant of the lexicographic product of two graphs $G_1$ and $G_2$, we must obtain bounds on the distances between any two arbitrary points in $G_1\circ G_2$. Besides, we study the geodesics in $G_1\circ G_2$, relating them with the geodesics in $G_1$. The lemmas of this section provide these results.

We will use the lexicographic product definition given in \cite{IK00}.

\begin{definition}\label{d:lexprod}
Let $G_1=(V(G_1),E(G_1))$ and $G_2=(V(G_2),E(G_2))$ be two graphs. The lexicographic product $G_1\circ G_2$ of $G_1$ and $G_2$ has $V(G_1) \times V(G_2)$ as vertex set, so that two distinct vertices $(u_1,v_1)$ and $(u_2,v_2)$ of $G_1\circ G_2$ are adjacent if either $[u_1,u_2]\in E(G_1)$, or $u_1=u_2$ and $[v_1,v_2]\in E(G_2)$.
\end{definition}

Note that the lexicographic product of two graphs is not always commutative. We use the notation $(x,y)$ for the points of the graph $G_1 \circ G_2$ with $x \in V(G_1)$ or $y \in V(G_2)$. Otherwise, this notation can be ambiguous.

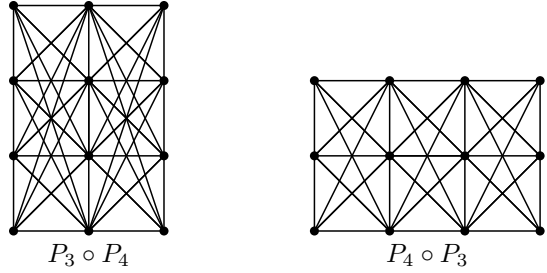
\begin{figure}[h]
\centering
\scalebox{1}
{\begin{pspicture}(-0.2,-0.7)(7.2,3.2)
\psline[linewidth=0.02cm,dotsize=0.07055555cm 2.5]{*-*}(0,0)(0,1)
\psline[linewidth=0.02cm,dotsize=0.07055555cm 2.5]{*-*}(0,2)(0,3)
\psline[linewidth=0.02cm,dotsize=0.07055555cm 2.5]{*-*}(1,0)(1,3)
\psline[linewidth=0.02cm,dotsize=0.07055555cm 2.5]{*-*}(1,1)(2,1)
\psline[linewidth=0.02cm,dotsize=0.07055555cm 2.5]{*-*}(1,2)(2,2)
\psline[linewidth=0.02cm,dotsize=0.07055555cm 2.5]{*-*}(2,0)(2,3)
\psline[linewidth=0.02cm,dotsize=0.07055555cm 2.5]{-}(0,0)(1,1)(0,2)(1,3)(0,3)(1,2)(0,1)(1,0)(0,0)
\psline[linewidth=0.02cm,dotsize=0.07055555cm 2.5]{-}(1,0)(2,1)(1,2)(2,3)(1,3)(2,2)(1,1)(2,0)(1,0)
\psline[linewidth=0.02cm,dotsize=0.07055555cm 2.5]{-}(0,1)(1,1)(1,2)(0,2)(0,1)
\psline[linewidth=0.02cm,dotsize=0.07055555cm 2.5]{-}(0,0)(1,1)(2,2)(1,3)(0,2)(1,1)(2,0)
\psline[linewidth=0.02cm,dotsize=0.07055555cm 2.5]{-}(0,3)(1,2)(2,1)(1,0)(0,1)(1,2)(2,3)
\psline[linewidth=0.02cm,dotsize=0.07055555cm 2.5]{-}(0,0)(1,2)(2,0)
\psline[linewidth=0.02cm,dotsize=0.07055555cm 2.5]{-}(0,2)(1,0)(2,2)
\psline[linewidth=0.02cm,dotsize=0.07055555cm 2.5]{-}(0,1)(1,3)(2,1)
\psline[linewidth=0.02cm,dotsize=0.07055555cm 2.5]{-}(0,3)(1,1)(2,3)
\psline[linewidth=0.02cm,dotsize=0.07055555cm 2.5]{-}(0,0)(1,3)(2,0)
\psline[linewidth=0.02cm,dotsize=0.07055555cm 2.5]{-}(0,3)(1,0)(2,3)
\uput[270](1,0){$P_3\circ P_4$}
%%%%%%
\psline[linewidth=0.02cm,dotsize=0.07055555cm 2.5]{*-*}(4,0)(5,0)
\psline[linewidth=0.02cm,dotsize=0.07055555cm 2.5]{*-*}(6,0)(7,0)
\psline[linewidth=0.02cm,dotsize=0.07055555cm 2.5]{*-*}(4,1)(7,1)
\psline[linewidth=0.02cm,dotsize=0.07055555cm 2.5]{*-*}(5,1)(5,2)
\psline[linewidth=0.02cm,dotsize=0.07055555cm 2.5]{*-*}(6,1)(6,2)
\psline[linewidth=0.02cm,dotsize=0.07055555cm 2.5]{*-*}(4,2)(7,2)
\psline[linewidth=0.02cm,dotsize=0.07055555cm 2.5]{-}(4,0)(5,1)(6,0)(7,1)(7,0)(6,1)(5,0)(4,1)(4,0)
\psline[linewidth=0.02cm,dotsize=0.07055555cm 2.5]{-}(4,1)(5,2)(6,1)(7,2)(7,1)(6,2)(5,1)(4,2)(4,1)
\psline[linewidth=0.02cm,dotsize=0.07055555cm 2.5]{-}(5,0)(5,1)(6,1)(6,0)(5,0)
\psline[linewidth=0.02cm,dotsize=0.07055555cm 2.5]{-}(4,0)(5,1)(6,2)(7,1)(6,0)(5,1)(4,2)
\psline[linewidth=0.02cm,dotsize=0.07055555cm 2.5]{-}(7,0)(6,1)(5,2)(4,1)(5,0)(6,1)(7,2)
\psline[linewidth=0.02cm,dotsize=0.07055555cm 2.5]{-}(4,0)(5,2)(6,0)(7,2)
\psline[linewidth=0.02cm,dotsize=0.07055555cm 2.5]{-}(4,2)(5,0)(6,2)(7,0)
\uput[270](5.5,0){$P_4\circ P_3$}
\end{pspicture}}
\caption{Non commutative lexicographic product of two graphs ($P_3\circ P_4 \not\simeq P_4\circ P_3$).} \label{fig:LexProd}
\end{figure}

\begin{remark}\label{r:remark1}
The Cartesian and the strong product of two graphs are subgraphs of the lexicographic product of two graphs, \emph{i.e.}, $G_1 \Box G_2\subseteq G_1 \boxtimes G_2\subseteq G_1\circ G_2$.
\end{remark}

Along this work by \emph{trivial graph} we mean a graph having just a single vertex, and we denote it by $E_1$.
If $G_1$ and $G_2$  are isomorphic, then we write $G_1 \simeq G_2$.

\begin{remark}\label{r:remark2}
Let $G$ be any graph. Then $G \circ E_1 \simeq G$ and $E_1 \circ G \simeq G$.
\end{remark}

In what follows we denote by $\pi$ the projection $\pi:G_1\circ G_2\rightarrow G_1$. The following result allows to compute the distance between any two vertices of $G_1\circ G_2$.

\begin{lemma}\label{l:dist}
 Let $G_1$ be a non-trivial graph and $G_2$ any graph and $(u,v)$, $(u',v')$ two vertices in $G_1\circ G_2$. Then
  \[d_{G_1\circ G_2}((u,v),(u',v'))=\left\{
\begin{array}{ll}
\min\{2,d_{G_2}(v,v')\},\quad &\mbox{if }\ u=u',\\

d_{G_1}(u,u'),\quad &\mbox{if }\ u \neq u'.
\end{array}
\right.
\]
 \end{lemma}

 \begin{proof}
 Assume first that $u=u'$, thus $(u,v), (u,v') \in V\big(\{u\} \circ G_2\big)$. If $d_{G_2}(v,v') \leq 2$ then
 $d_{G_1\circ G_2}\big( (u,v),(u,v')\big)=d_{G_2}(v,v')$ since a path in $G_1\circ G_2$ joining $(u,v)$ and $(u,v')$ which is not contained in $\{u\} \circ G_2$ has a vertex out of $\{u\} \circ G_2$, and so, its length is at least $2$.
 If $d_{G_2}(v,v') > 2$ then
 $$d_{G_1\circ G_2}((u,v),(u,v'))= d_{G_1\circ G_2}((u,v),\{w\} \circ G_2) + d_{G_1\circ G_2}(\{w\} \circ G_2,(u,v'))=2,$$
 where $[u,w] \in E(G_1)$.

Assume now that $u\neq u'$. If $\g:=[uu']$  is a geodesic in $G_1$ joining the points $u$ and $u'$ with $L(\g)=k$, then there exist vertices $A_{1},\ldots ,A_{k-1}$ in $\g \setminus \{u,u'\}$. Without loss of generality we can assume that $\g$ meets $A_{1},\ldots,A_{k-1}$ in this order. If we fix $v_0\in V(G_2)$, then
 $$d_{G_1\circ G_2}((u,v),(u',v'))\leq d_{G_1\circ G_2}((u,v),(A_1,v_0)) + \ldots + d_{G_1\circ G_2}((A_{k-1},v_0),(u',v')) = k.$$

   If $d_{G_1\circ G_2}((u,v) , (u',v')) < k$, then there exists a geodesic $\G$ in $G_1\circ G_2$ joining $(u,v)$ and $(u',v')$ with $L(\G) = r < k$.
   Denote by $B_{1},\ldots,B_{r-1}$ the vertices in $\G \setminus \{(u,v),(u',v')\}$. Without loss of generality we can assume that $\G$ meets $B_{1},\ldots,B_{r-1}$ in this order.
   Then we have
   $$\G:= [(u,v),B_{1}] \bigcup \left\{ \bigcup_{j=1}^{r-2} [B_{j}, B_{j+1}]\right\} \bigcup [ B_{r-1},(u',v')].$$
   By Definition \ref{d:lexprod},
   $$\g_1:= [u,\pi(B_{1})] \bigcup \left\{ \bigcup_{j=1}^{r-2} [\pi(B_{j}),\pi(B_{j+1})]\right\} \bigcup [\pi(B_{r-1}),u']$$
   is a path joining $u$ and $u'$ in $G_1$ such that $L(\g_1) \leq L(\G) < L(\g)$.
   This is a contradiction, thus
   $$d_{G_1\circ G_2}((u,v),(u',v'))=d_{G_1}(u,u').$$
 \end{proof}

Let $X$ be a metric space, $Y$ a non-empty subset of $X$ and $\varepsilon$ a positive number.
We call $\varepsilon$-\emph{neighborhood}
of $Y$ in $X$, denoted by $\mathcal{V}_{\varepsilon}(Y)$ to the set $\{x\in X: d_X(x,Y)\leq\varepsilon\}$.

\begin{lemma}\label{l:vecindad}
 Let  $G_1$ be a non-trivial graph and $G_2$ any  graph. Then $G_1 \circ G_2 \subseteq \mathcal{V}_{3/2}{(G_1 \circ \{v\})}$ for every $v \in V(G_2)$.
\end{lemma}

\begin{proof}
Let $p$ be any point of $G_1 \circ G_2$. If $p \in V(G_1 \circ G_2)$, then consider any $u_0\in V(G_1)$ such that $[\pi(p),u_0]\in E(G_1)$. Definition \ref{d:lexprod} gives $d_{G_1 \circ G_2}(p,G_1\circ\{v\})\leq d_{G_1 \circ G_2}(p,(u_0,v)) =1$ for every $v \in V(G_2)$ since $G_1$ is non-trivial.
Assume that  $p \notin V(G_1 \circ G_2)$. Let $A\in V(G_1 \circ G_2)$ with $d_{G_1 \circ G_2}(p,A)\leq 1/2$. Hence, we have
$$d_{G_1 \circ G_2}(p,G_1 \circ \{v\}) \leq d_{G_1 \circ G_2}(p, A) + d_{G_1 \circ G_2}(A, G_1 \circ \{v\}) \leq 3/2.$$
\end{proof}

\begin{lemma} \label{l:geodesicas}
Let $y_1, y_2$ be any points in $G_2$ with $d_{G_2}(y_1,y_2) \leq 5/2$ and  $x_0$ a fixed vertex in $G_1$. Then $\g:=\{x_0\} \times [y_1 y_2]$ is a geodesic in $G_1 \circ G_2$ joining the points $(x_0,y_1)$ and $(x_0,y_2)$.
\end{lemma}

\begin{proof}
If $G_1$ is the trivial graph, then $G_1\circ G_2 \simeq G_2$ and we have the result. Assume that $G_1$ is a non-trivial graph.
Seeking for a contradiction assume that $\g$ is not a geodesic in $G_1 \circ G_2$.
Therefore, there is a geodesic $\G$ in $G_1 \circ G_2$ joining $(x_0,y_1)$ and $(x_0,y_2)$ which is not contained in $\{x_0\} \circ G_2$. Hence, $\G$ has a vertex $V$ outside of $\{x_0\} \circ G_2$; thus, we have $2\le L(\G)< L(\g)\le 5/2$. We have
$$\G=[(x_0,y_1) (x_0,B_1)]\cup[(x_0,B_1),V]\cup[V,(x_0,B_2)]\cup[(x_0,B_2)(x_0,y_2)],$$
where $B_i$ is a closest vertex to $y_i$ in $G_2$, for $i=1,2$. Since $\g\cup\G$ contains a cycle $C$ with $(x_0,B_1),(x_0,B_2)\in C$ and  $L(\g) + L(\G) < 5$ we have $L(C)\leq 4$ and $d_{G_2}(B_1,B_2)\le 2$, and so, we obtain
\begin{eqnarray*}
d_{G_2}(y_1,y_2)\!\!\!\! &\le&\!\!\!\! d_{G_2}(y_1,B_1) + d_{G_2}(B_1,B_2) + d_{G_2}(B_2,y_2)\\
 &\le& \!\!\!\!  d_{G_2}(y_1,B_1) + 2 + d_{G_2}(B_2,y_2) = L(\G) < L(\g)=d_{G_2}(y_1,y_2).
\end{eqnarray*}
This is the contradiction we were looking for, and so, $\g$ is a geodesic in $G_1 \circ G_2$.
\end{proof}

\begin{corollary}\label{c:geod3}
Let  $G_1$ be a non-trivial graph and $G_2$ any  graph, $y_1, y_2$ any points in $G_2$ with $d_{G_2}(y_1,y_2) > 3$ and  $x_0$ a fixed vertex in $G_1$.
Then  $\{x_0\} \times [y_1 y_2]$ is not a geodesic in $G_1 \circ G_2$.
\end{corollary}

\begin{proof}
Let $B_i$ be the closest vertex to $y_i$ in $G_2$, for $i=1,2$. Since $G_1$ is a non-trivial graph there is a vertex $u_0\in V(G_1)$ such that $[x_0,u_0]\in E(G_1)$. For any fixed $v_0\in V(G_2)$ we have
$$\G:=[(x_0,y_1) (x_0,B_1)]\cup[(x_0,B_1),(u_0,v_0)]\cup[(u_0,v_0),(x_0,B_2)]\cup[(x_0,B_2)(x_0,y_2)]$$
is a path in $G_1\circ G_2$ joining $(x_0,y_1)$ and $(x_0,y_2)$. Besides, since $d_{G_2}(y_1,B_1)\le 1/2$ and $d_{G_2}(y_2,B_2)\le 1/2$ we have $L(\G)\le 3 < d_{G_2}(y_1,y_2)=L(\{x_0\}\times [y_1y_2])$.
\end{proof}

\begin{remark}\label{r:midpoints}
Let $y_1, y_2$ be two midpoints in any graph $G_2$ with $d_{G_2}(y_1,y_2) = 3$ and  $x_0$ a fixed vertex in any graph $G_1$.
Then  $\{x_0\} \times [y_1 y_2]$ is a geodesic in $G_1 \circ G_2$ joining $(x_0,y_1)$ and $(x_0,y_2)$.
 \end{remark}

\begin{lemma}\label{l:+3vert}
Let  $G_1$ be a non-trivial graph and $G_2$ be any  graph. If $\g$ is a geodesic in $G_1 \circ G_2$ joining $x$ and $y$ with $L(\g) > 3$, then $\pi(\g)$ contains at least three vertices in $G_1$.

Furthermore, if $\s$ is a path in $G_1 \circ G_2$ joining $x$ and $y$, then $\pi(\s)$ contains at least three vertices in $G_1$.
\end{lemma}

\begin{proof}
Since $L(\g)>3$ then $\g$ contains at least three vertices in $G_1 \circ G_2$. Let $V_1$ and $V_2$ be the closest vertices to $x$ and $y$ in $\g$, respectively.
Seeking for a contradiction assume that $\pi(\g)$ contains either one or two vertices in $G_1$.
Since $G_1$ is a non-trivial graph and $\pi(\g)$ contains at most two vertices, Lemma \ref{l:dist} gives that $d_{G_1 \circ G_2}(V_1,V_2)=2$ and $\pi(V_1)= \pi(V_2)$.
Furthermore, since $L(\g)>3$ we have either $d_{G_1 \circ G_2}(x,V_1)>1/2$ or $d_{G_1 \circ G_2}(y,V_2)>1/2$. Without loss of generality we can assume that $d_{G_1 \circ G_2}(x,V_1)>1/2$.
Let $W$ be the vertex in $G_1\circ G_2$ with $x$ in the edge $[V_1,W]$.
Then $d_{G_1 \circ G_2}(x,W)< 1/2 < d_{G_1 \circ G_2}(x,V_1)$.
Consider now a path $\g_1:=[x W]\cup [W V_2] \cup [V_2 y]$ joining $x$ and $y$ in $G_1 \circ G_2$. Hence, $L(\g_1)<L(\g)$ since $d_{G_1 \circ G_2}(W,V_2) \leq 2$. This is the contradiction we were looking for, and then $\pi(\g)$ contains at least three vertices in $G_1$.
Finally, since $L(\s)\ge L(\g)$ and $\pi(\g)$ contains at least three vertices, the proof is straightforward.
\end{proof}

\begin{lemma}\label{l:proy}
Let  $G_1$ be a non-trivial graph and $G_2$ be any  graph. Consider a geodesic $\g$ in $G_1 \circ G_2$ joining $x$ and $y$. If $L(\g) > 3$, then $\pi(\g)$ is a geodesic in $G_1$ joining $\pi(x)$ and $\pi(y)$. Besides, if $L(\g) = 3$ then $\pi(\g)$ contains a geodesic in $G_1$ joining $\pi(x)$ and $\pi(y)$.
\end{lemma}

\begin{proof}
Assume first that $L(\g)>3$. By Lemma \ref{l:+3vert}, $\pi(\g)$ contains at least three vertices in $G_1$. Denote by $V_{1},\ldots,V_{r}$ the vertices of $G_1\circ G_2$ in $\g$ with $r \geq 3$, and $v_{1},\ldots,v_{r}$ their projections  in $G_1$ (there are at least three different vertices). Without loss of generality we can assume that $\g$ meet $V_{1},\ldots,V_{r}$ in this order. Let $V_{1}',V_{r}'$ be two vertices in $G_1\circ G_2$ such that $x\in [V_{1}',V_1]$ and $y\in [V_{r}',V_r]$, and denote by $v_{1}',v_{r}'$ their projections in $G_1$, respectively.
Since $d_{G_1\circ G_2}( V_{1} , V_{r} )\geq 2$ and $d_{G_1\circ G_2}( x , y )\geq 3$, Lemma \ref{l:dist} gives $d_{G_1}\big( \{v_{1},v_1'\} , \{v_{r},v_r'\} \big)\geq 2$.

Seeking for a contradiction assume that there is a geodesic $\G$ in $G_1$ joining $\pi(x)$ and $\pi(y)$ with length less than $L(\pi(\g))$.
Let us consider $v_i^*:=\{v_{i},v_i'\}\cap \G$ and $V_i^*\in\{V_i,V_i'\}$ with $\pi(V_i^*)=v_i^*$ for $i\in\{1,r\}$.
Now, we have three cases.
\begin{enumerate}
\item $\pi(x)\neq v_1$ and $\pi(y)\neq v_r$. Then $\pi(x)\in [v_{1}',v_1]$ and $\pi(y)\in [v_{r}',v_r]$. Let $\g_1:=[x V_{1}^{*}]\cup [V_{1}^{*}V_{r}^{*}]\cup [V_{r}^{*}y]\subset G_1\circ G_2$.
Since $d_{G_1}( v_1^* , v_{r}^* )\geq 2$, Lemma \ref{l:dist} gives $d_{G_1\circ G_2}(V_{1}^{*},V_{r}^{*})=d_{G_1}(v_{1}^{*},v_{r}^{*})$, and so $L(\g_1)= L(\G)<L(\pi(\g))\le L(\g)$.
This is the contradiction we were looking for, and so, $\pi(\g)$ is a geodesic in $G_1$ joining $\pi(x)$ and $\pi(y)$.

\item $\pi(x)= v_1$ and $\pi(y)\neq v_r$ or $\pi(x)\neq v_1$ and $\pi(y)= v_r$. By symmetry, we can assume $\pi(x)=v_1$ and $\pi(y)\neq v_r$. Then $\pi(y)\in [v_{r}',v_r]$ and $d_{G_1\circ G_2}(x,V_1)\leq 1/2$. Let $\g_1:=[x V_{1}]\cup [V_{1}V_{r}^{*}]\cup [V_{r}^{*}y]\subset G_1\circ G_2$. Since $d_{G_1}( v_{1} , v_{r}^* )\geq 2$, Lemma \ref{l:dist} gives $d_{G_1\circ G_2}(V_{1},V_{r}^{*})=d_{G_1}(v_{1},v_{r}^{*})$, and so $L(\g_1)=L(\G)+L([xV_1])<L(\pi(\g))+L([xV_1])\le L(\g)$. This is the contradiction we were looking for, and so, $\pi(\g)$ is a geodesic in $G_1$ joining $\pi(x)$ and $\pi(y)$.

\item $\pi(x)= v_1$ and $\pi(y)= v_r$.
Then $\pi(\g)=\pi([V_1V_r])$. Since $d_{G_1}( v_{1} , v_{r} )\geq 2$, Lemma \ref{l:dist} gives $d_{G_1\circ G_2}(V_{1},V_{r})=d_{G_1}(v_{1},v_{r})$. Then $L\big(\pi(\g)\big)=d_{G_1}(v_1,v_r)$, and $\pi(\g)$ is a geodesic in $G_1$ joining $\pi(x)$ and $\pi(y)$.
\end{enumerate}

Assume now that $L(\g)=3$. Then $\pi(\g)$ contains either one, two, three or four vertices in $G_1$.

If $\pi(\g)$ contains a single vertex in $G_1$, then $\g$ is contained in $\{v\}\circ G_2$ for some $v\in V(G_1)$. Thus, $\pi(\g)=v$ is a geodesic in $G_1$ joining $\pi(x)$ with $\pi(y)$.

If $\pi(\g)$ contains exactly two vertices in $G_1$, then $x,y$ are midpoints  of edges and $\pi(x)=\pi(y)$.

If $\pi(\g)$ contains three or four vertices in $G_1$, then $\pi(\g)$ contains a geodesic in $G_1$ joining $\pi(x)$ and $\pi(y)$, and the argument used in the proof of the case $L(\g)>3$ gives that $\pi(\g)$ is a geodesic.
\end{proof}

\begin{remark}\label{r:proy3}
Let $\g$ be a geodesic in $G_1 \circ G_2$ joining $x$ and $y$. If $L(\g) = 3$ and $\pi(\g)$ is not a geodesic in $G_1$ joining $\pi(x)$ and $\pi(y)$, then $x,y$ are midpoints of edges, $\pi(x)=\pi(y)\in V(G_1)$ and $\diam (\pi(\g))=1$.
\end{remark}

\begin{definition}\label{def:diam}
 The diameter of the vertices of the graph $G$, denoted by $\diam V(G)$, is defined as,
    $$\diam V(G):= \sup \{d_{G}(u,v): u,v\in V(G)\},$$
 and the diameter of the graph $G$, denoted by $\diam G$, is defined as,
$$\diam G:= \sup \{d_{G}(x,y): x,y\in G\}.$$
\end{definition}

\begin{corollary}\label{c:proy3}
Let $\g$ be a geodesic in $G_1 \circ G_2$ joining $x$ and $y$. If $\pi(\g)$ is not a geodesic in $G_1$ joining $\pi(x)$ and $\pi(y)$, then $\diam \big(\pi(\g)\big)< 3$.
\end{corollary}

Notice that, if $\g$ is a geodesic  in $G_1\circ G_2$ joining the points $x$ and $y$, then it is possible that $\pi(\g)$ does not contain a geodesic  in $G_1$ joining the points $\pi(x)$ and $\pi(y)$, as the following example shows.

\begin{example}
Consider $G_1$ as the \emph{cycle graph} $C_3$ with vertices $\{v_1,v_2,v_3\}$ and $G_2$ as the \emph{path graph} $P_3$ with vertices $\{w_1,w_2,w_3\}$ and $E(G_2)=\{[w_1,w_2],[w_2,w_3]\}$.
Let $x$ and $y$ be the midpoints of edges $[(v_1,w_1),(v_2,w_1)]$ and $[(v_1,w_3),(v_3,w_3)]$, respectively.
We have that $\g:= [x (v_2,w_1)]\cup [(v_2,w_1),(v_3,w_3)] \cup [(v_3,w_3) y]$ is a geodesic in $G_1\circ G_2$ joining $x$ and $y$, but $\pi(\g)=[\pi(x) v_2] \cup [v_2, v_3] \cup [v_3 \pi(y)]$ does not contain the geodesic in $G_1$ joining $\pi(x)$ and $\pi(y)$ (note that this geodesic is $[\pi(x) v_1] \cup [v_1 \pi(y)]$).
\end{example}

%SECTION  3------------------------------------------------------------------

\
\section{Hiperbolicity in lexicographic products}

Some bounds for the hyperbolicity constant of the lexicographic product of two graphs are studied in this section.
These bounds allow to prove Theorem \ref{t:Caracterizacion 1}, which characterizes the hyperbolic lexicographic products of two graphs.

We say that a subgraph  $\G$ of $G$ is \emph{isometric} if $d_{\G}(x,y)=d_{G}(x,y)$ for every $x,y\in \G$.
The following result which appears in \cite[Lemma 5]{RSVV} will be useful.

\begin{lemma}\label{l:subgraph}
 If $\G$ is an isometric subgraph of $G$, then $\d(\G) \le \d(G)$.
\end{lemma}

The next theorem shows an important qualitative result: if $G_1$ is not hyperbolic then $G_1 \circ G_2$ is not hyperbolic.

\begin{theorem}\label{t:sublex}
 Let $G_1$ and $G_2$ two graphs, then $\d(G_1) \leq \d(G_1 \circ G_2).$
\end{theorem}

\begin{proof}
Since $G_1 \circ \{y\}$ is an isometric subgraph of $G_1 \circ G_2$ for every $y \in V(G_2)$, Lemma \ref{l:subgraph} gives the result.
\end{proof}

Example \ref{ex:Cn_P2} shows that the equality in Theorem \ref{t:sublex} is attained: $\d(C_n)=\d(C_n\circ P_2)$ for $n\geq 5$.

Note that the strong product graph $G\boxtimes P_2$ is isomorphic to $G\circ P_2$ for any graph $G$.
We recall that $\d(P_n)=0$ since the path graph $P_n$ is a tree; besides, it is well known that the hyperbolicity constant of the cycle graph $C_n$ is $n/4$, see \cite[Theorem 11]{RSVV}.
The following results which appear in \cite{CCCR} give the hyperbolicity constant of some lexicographic product graphs.

\begin{example}\label{ex:Pn_P2}
Let $P_n$ be the path graph with $n\geq 2$. Then
  \[\d(P_n\circ P_2)=\left\{
\begin{array}{ll}
1,\quad &\mbox{if }\ n=2,\\

5/4,\quad &\mbox{if }\ n=3,\\

3/2,\quad &\mbox{if }\ n\ge 4.
\end{array}
\right.
\]
\end{example}

\begin{example}\label{ex:Cn_P2}
Let $C_n$ be the cycle graph with $n\geq 3$. Then

  \[\d(C_n\circ P_2)=\left\{
\begin{array}{ll}
1, \quad &\mbox{if }\ n=3,\\
5/4,\quad &\mbox{if }\ n=4,\\
n/4,\quad &\mbox{if }\ n\ge5.
\end{array}
\right.
\]
\end{example}

\begin{example}\label{ex:complete}
Let $K_m, K_n$ be the complete graphs with $m,n$ vertices, respectively, and $m,n\ge 2$. Then $K_m\circ K_n$ is isomorphic to $K_{mn}$ and $\d(K_m\circ K_n)=1$.
\end{example}

\begin{proposition}\label{p:IsomLex}
  Let  $G_1$ be a non-trivial graph and $G_2$ any  graph. Consider isometric subgraphs $\G_1,\G_2$ of $G_1,G_2$, respectively, with $\G_1$ non-trivial. Then $\G_1\circ \G_2$ is an isometric subgraph to $G_1\circ G_2$.
\end{proposition}

Note that taking $\G_1$ as a trivial graph, $\G_1\circ \G_2$ is not an isometric subgraph to $G_1\circ G_2$ if $\diam V(\G_2) \ge3$.

\begin{proof}
Since $\G_1\circ \G_2$ is a subgraph of $G_1\circ G_2$, we have $d_{\G_1\circ \G_2}(x,y)\geq d_{G_1\circ G_2}(x,y)$ for every $x,y\in \G_1\circ \G_2$.
Let $x,y$ be any points of $\G_1\circ \G_2$. If $x,y\in V(\G_1\circ \G_2)$ then by Lemma \ref{l:dist} we have $d_{G_1\circ G_2}(x,y)=d_{\G_1\circ \G_2}(x,y)$ and we obtain the result.
Without loss of generality we can assume that $x,y\notin V(\G_1\circ \G_2)$.
Let $A_1,A_2,B_1,B_2\in V(\G_1 \circ \G_2)$ with $x\in [A_1,A_2]$, $y\in [B_1,B_2]$.
Consider a geodesic $\g$ in $G_1\circ G_2$ joining $x$ and $y$ with $\g:=[xA_i]\cup [A_iB_j]\cup [B_jy]$ for some $i,j\in \{1,2\}$.
Then
$$d_{\G_1\circ \G_2}(x,y)\leq d_{\G_1\circ \G_2}(x,A_i)+d_{\G_1\circ \G_2}(A_i,B_j)+d_{\G_1\circ \G_2}(B_j,y)=d_{G_1\circ G_2}(x,y).$$
Thus, $d_{G_1\circ G_2}(x,y)=d_{\G_1\circ \G_2}(x,y)$.
\end{proof}

\begin{theorem}\label{t:HypIsomLex}
Let  $G_1$ be a non-trivial graph and $G_2$ any graph. Then
$$\d(G_1\circ G_2)=\max\{ \d(\G_1\circ \G_2) : \G_i \text{ is isometric to } G_i \text{ for } i=1,2 \text{ and } \G_1 \text{ non-trivial}\}.$$
\end{theorem}

\begin{proof}
By Lemma \ref{l:subgraph} and Proposition \ref{p:IsomLex} we have $\d(G_1\circ G_2)\ge \d(\G_1\circ \G_2)$ for any $\G_1,\G_2$. Besides, since any graph is an isometric subgraph of itself we obtain the equality by taking $\G_1=G_1$ and $\G_2=G_2$.
\end{proof}

The next results will be useful.

\begin{theorem}[Theorem 8 in \cite{RSVV}]
\label{t:diameter1} In any graph $G$ the inequality $\d(G)\le\frac12\,\diam G$ holds and
it is sharp.
\end{theorem}

Denote by $J(G)$ the set of vertices and midpoints of edges in $G$.
As usual, by \emph{cycle} we mean a simple closed curve, \emph{i.e.}, a path with different vertices,
unless the last one, which is equal to the first vertex.

\begin{theorem}[Theorem 2.6 in \cite{BRS}]
\label{t:multk/4}
For every hyperbolic graph $G$, $\d(G)$ is a multiple of $1/4$.
\end{theorem}

\begin{theorem}[Theorem 2.7 in \cite{BRS}]
\label{t:TrianVMp}
For any hyperbolic graph $G$, there exists
a geodesic triangle $T = \{x, y, z\}$ that is a cycle with $x, y, z \in J(G)$ and $\d(T) = \d(G)$.
\end{theorem}

\begin{theorem}\label{t:Cota 1}
If $G_1$ and $G_2$ are non-trivial graphs, then $\d(G_1\circ G_2)\ge 1$.
\end{theorem}

\begin{proof}
Since $G_i$ is a non-trivial graph there is a subgraph $P_{2}^{i}$ in $G_i$ isomorphic to an edge, for $i=1,2$. Hence, by Example \ref{ex:Pn_P2} and Theorem \ref{t:HypIsomLex} we have $\d(G_1\circ G_2)\ge \d(P_{2}^{1}\circ P_{2}^{2})=1$.
\end{proof}

\begin{theorem}\label{t:Cota 5/4 G1}
Let $G_2$ be any non-trivial graph and $G_1$ any graph.
If $\diam V(G_1)= 2$, then $\d(G_1\circ G_2)\ge 5/4$.
If $\diam V(G_1)\ge 3$, then $\d(G_1\circ G_2)\ge 3/2$.
\end{theorem}

\begin{proof}
Assume that $\diam V(G_1)= 2$.
Since $G_2$ is a non-trivial graph there is a subgraph $P_2$ in $G_2$ isomorphic to an edge. Besides, since $\diam V(G_1)= 2$ then there is an isometric subgraph in $G_1$ isomorphic to a path $P_3$ with $3$ vertices. Example \ref{ex:Pn_P2} and Theorem \ref {t:HypIsomLex} give $5/4=\d(P_3\circ P_2)\leq \d(G_1 \circ G_2)$.

If $\diam V(G_1)\ge 3$, then a similar argument replacing $P_3$ by $P_4$ gives $\d(G_1\circ G_2)\ge 3/2$.
\end{proof}

\begin{theorem}\label{t:Cota 5/4 G2}
If $G_1$ is any non-trivial graph and $G_2$ is any graph with $\diam G_2 > 2$, then $\d(G_1\circ G_2)\ge 5/4$.
\end{theorem}

\begin{proof}
Since $\diam G_2\ge 5/2$ we have that there exist a midpoint $x\in J(G_2)\setminus V(G_2)$ and a vertex $y\in V(G_2)$ such that $d_{G_2}(x,y)=5/2$. Hence, by Lemma \ref{l:geodesicas} we have that $\g_1:=\{v_0\}\times [xy]$ is a geodesic in $G_1\circ G_2$ joining the points $(v_0,x)$ and $(v_0,y)$ for some $v_0\in V(G_1)$. Without loss of generality we can assume that $(v_0,x)\in [A_1,A_2]$ such that $A_1\in \g_1$. Denote it by $\g_2:=[(v_0,x)A_2]\cup [A_2W]\cup [W(v_0,y)]$ where $W\in V(\{v_1\}\circ G_2)$ with $[v_0,v_1]\in E(G_1)$. Therefore, $L(\g_2)=5/2$ and $\g_2$ is a geodesic in $G_1\circ G_2$ joining the points $(v_0,x)$ and $(v_0,y)$. Now we have a geodesic bigon $B:=\{\g_1,\g_2\}$ in $G_1\circ G_2$. If $p$ is the midpoint of $\g_1$, then $d_{G_1\circ G_2}(p,\g_2)= 5/4$ and we conclude that $\d(G_1\circ G_2)\ge \d(B) = 5/4$.
\end{proof}

\begin{theorem}\label{t:CotaSup}
Let  $G_1$ be any non-trivial graph and $G_2$ any  graph. Then we have $\d(G_1 \circ G_2) \leq \d(G_1)+3/2$.
\end{theorem}

\begin{proof}
If $G_1$ is not hyperbolic, then $\d(G_1)=\infty$, and so, Theorem \ref{t:sublex} gives the result (with equality). Assume now that $G_1$ is hyperbolic.
By Theorem \ref{t:TrianVMp} it suffices to consider geodesic triangles $T=\{x,y,z\}$ in $G_1 \circ G_2$ that are cycles with $x,y,z \in J(G_1 \circ G_2)$. Let  $\g_1 := [x y]$, $\g_2 := [y z]$ and $\g_3 := [zx]$. It suffices to prove that $d_{G_1 \circ G_2}(p, \g_2 \cup \g_3) \leq \d(G_1) + 3/2$ for every $p \in \g_1$. If $d_{G_1\circ G_2}(p,\{x,y\})\leq 3/2$, then $d_{G_1 \circ G_2}( p , \g_2 \cup \g_3 ) \leq d_{G_1 \circ G_2}( p , \{x,  y\}) \leq 3/2$.

Assume that $d_{G_1 \circ G_2}(p, \{x,y\} ) > 3/2$; then $L(\g_1)>3$. Let $V_p := (v,w)$ be a closest vertex to $p$ in $\g_1$. Consider the canonical projection $\pi:G_1\circ G_2\longrightarrow G_1\circ \{w\}$. By Lemma \ref{l:proy}, $\pi(\g_1)$ is a geodesic in $G_1\circ \{w\}$ joining the points $\pi(x)$ and $\pi(y)$.

If $\pi(\g_2)$ and $\pi(\g_3)$ are geodesics in $G_1\circ \{w\}$, then there is a point $\a\in \pi(\g_2)\cup \pi(\g_3)$ such that $d_{G_1\circ \{w\}}(V_p,\a)\leq \d(G_1)$. Assume that $\a\in V\big(\pi(\g_2)\cup \pi(\g_3)\big)$. Since $L(\g_1) > 3$ and $\g_2\cup \g_3$ joins $x$ and $y$, by Lemma \ref{l:+3vert}, $\pi(\g_2)\cup \pi(\g_3)$ contains at least three vertices; hence, there exists a vertex $(v_{\a},w)\in V(\pi(\g_2)\cup \pi(\g_3))$ such that $[\a,(v_{\a},w)]\in E(G_1\circ \{w\})$. Let $V_{\a}$ be a vertex in $\big(\{v_{\a}\}\circ G_2\big)\cap \big(\g_2\cup\g_3\big)$. Thus, $[\a,V_{\a}]\in E(G_1\circ G_2)$ and
$$d_{G_1\circ G_2}(p,\g_2\cup \g_3)\leq d_{G_1\circ G_2}(p,V_p)+d_{G_1\circ \{w\}}(V_p,\a)+d_{G_1\circ G_2}(\a,V_{\a})\leq \d(G_1)+3/2.$$
If $\a\notin V(\pi(\g_2)\cup \pi(\g_3))$, then $\a\in \{\pi(x),\pi(y)\}$ and $\a$ is a midpoint in $G_1\circ \{w\}$. Without loss of generality we can assume that $\a=\pi(x)$ and, consequently, $x$ is a midpoint in $G_1\circ G_2$. Let $V_{x}$ be the closest vertex to $x$ in $\g_2\cup \g_3$ and $v_x$ the closest vertex to $\pi(x)$ in $\pi(\g_1)$. Hence, $[V_x,v_x]\in E(G_1\circ G_2)$, $d_{G_1\circ \{w\}}(V_p,v_x)\leq \d(G_1)- 1/2$ and
$$d_{G_1\circ G_2}(p,\g_2\cup \g_3)\leq d_{G_1\circ G_2}(p,V_p)+d_{G_1\circ \{w\}}(V_p,v_x)+d_{G_1\circ G_2}(v_x,V_{x})\leq \d(G_1)+1.$$

If $\pi(\g_2)$ and $\pi(\g_3)$ are not geodesics in $G_1\circ \{w\}$, then there is a point $\a\in [\pi(x)\pi(z)]\cup [\pi(z)\pi(y)]$ such that $d_{G_1\circ \{w\}}(V_p,\a)\leq \d(G_1)$. Notice that, if $\a$ is not a vertex in $G_1\circ \{w\}$ then we repeat the previous argument and obtain the result. Assume now that $\a\in V([\pi(x)\pi(z)]\cup [\pi(z)\pi(y)])$; by symmetry, we can assume that $\a\in V([\pi(x)\pi(z)])$. If $\a\in \pi(\g_2)\cup \pi(\g_3)$, then the previous argument gives $d_{G_1\circ G_2}(p, \g_2 \cup \g_3)\leq \d(G_1)+3/2$. Assume now that $\a \notin \pi(\g_2)\cup \pi(\g_3)$.  Seeking for a contradiction assume that there is not a vertex $(v_{\a},w)\in V(\pi(\g_2)\cup \pi(\g_3))$ such that $[\a,(v_{\a},w)]\in E(G_1\circ \{w\})$.
Then $d_{G_1\circ \{w\}}(\a , V(\pi(\g_2)\cup\pi(\g_3)))\ge 2$; hence, $d_{G_1\circ \{w\}}(\a,\pi(x))\ge3/2$ and $d_{G_1\circ \{w\}}(\a,\pi(z))\ge3/2$.
However, by Corollary \ref{c:proy3} we have $d_{G_1\circ \{w\}}(\pi(x),\pi(z))=d_{G_1\circ \{w\}}(\pi(x),\a)+d_{G_1\circ \{w\}}(\a,\pi(z))<3$, which is a contradiction. Therefore, there exists a vertex $(v_{\a},w)\in V(\pi(\g_2)\cup \pi(\g_3))$ such that $[\a,(v_{\a},w)]\in E(G_1\circ \{w\})$. Let $V_{\a}$ be a vertex in $\big(\{v_{\a}\}\circ G_2\big)\cap \big(\g_2\cup\g_3\big)$. Then $[\a,V_{\a}]\in E(G_1\circ G_2)$ and
$$d_{G_1\circ G_2}(p,\g_2\cup \g_3)\leq d_{G_1\circ G_2}(p,V_p)+d_{G_1\circ \{w\}}(V_p,\a)+d_{G_1\circ G_2}(\a,V_{\a})\leq \d(G_1)+3/2.$$

In both cases, $\pi(\g_2)$ is a geodesic in $G_1\circ \{w\}$ but $\pi(\g_3)$ is not a geodesic in $G_1\circ \{w\}$, and $\pi(\g_3)$ is a geodesic in $G_1\circ \{w\}$ but $\pi(\g_2)$ is not a geodesic in $G_1\circ \{w\}$, a similar argument gives the inequality.
\end{proof}

\begin{remark}\label{r:midpoint}
Let $G_1$ be any hyperbolic graph which is not a tree and let $G_2$ be any graph. The argument in the proof of Theorem \ref{t:CotaSup} gives that if $\d(G_1\circ G_2)= \d(G_1) + 3/2$ then there is a geodesic
triangle $T = \{x, y, z\}$ with $x, y, z \in J(G_1 \circ G_2)$ and a midpoint $p \in [xy]$ such that $d _{G_1\circ G_2} (p, [xz] \cup [zy]) = \d(G_1) + 3/2$. Besides, $d_{G_1\circ \{w\}}(V_p,[\pi(x)\pi(z)]\cup [\pi(z)\pi(y)])=\d(G_1)$ and the distance is attained in a vertex $\a\in [\pi(x)\pi(z)]\cup [\pi(z)\pi(y)]$.
\end{remark}

Example \ref{ex:Pn_P2} and Theorem \ref{th:hyp3/2_diam2} show that the equality in Theorem \ref{t:CotaSup} is attained.

We obtain the following consequence of Theorem \ref{t:sublex} and Theorem \ref{t:CotaSup}.

\begin{theorem}\label{t:InfSup}
Let  $G_1$ be any non-trivial graph and $G_2$ any  graph. Then
$$\d(G_1) \leq \d(G_1 \circ G_2) \leq \d(G_1)+3/2.$$
\end{theorem}

Theorems \ref{t:Cota 5/4 G1} and \ref{t:CotaSup} have the following consequence.

\begin{corollary}\label{c:infinity}
If $G_1$ is any infinite tree and $G_2$ is any non-trivial graph, then $\d(G_1\circ G_2)= 3/2$.
\end{corollary}

\begin{theorem}\label{t:Caracterizacion 1}
Let  $G_1$ be any non-trivial graph and $G_2$ any  graph. The lexicographic product $G_1 \circ G_2$ is hyperbolic if and only if $G_1$ is hyperbolic.
\end{theorem}

\begin{remark}\label{r:d_Trivial}
For any graph $G$ and the trivial graph $E_1$, the lexicographic product graph $E_1\circ G$ is hyperbolic if and only if $G$ is hyperbolic, since $\d(E_1\circ G)=\d(G)$.
This trivial result completes the characterization of hyperbolic lexicographic products.
\end{remark}

The following results allow to characterize the graphs for which the bound in Theorem \ref{t:CotaSup} is attained.

\begin{theorem}\label{th:equal+3/2}
Let $G_1$ be any hyperbolic  graph and let $G_2$ be any graph. If $\d( G_1\circ G_2)=\d(G_1)+3/2$, then $G_1$ is a tree, $G_2$ is a non-trivial graph and $\d(G_1\circ G_2)=3/2$.
\end{theorem}

\begin{proof}
Seeking for a contradiction assume that $G_1$ is not a tree (\emph{i.e.}, $\d(G_1)>0$).
By hypothesis $G_1\circ G_2$ is hyperbolic, thus, Theorem \ref{t:TrianVMp} and Remark \ref{r:midpoint} give that there is a geodesic triangle $T=\{x,y,z\}$ in $G_1 \circ G_2$ that is a cycle with $x,y,z \in J(G_1 \circ G_2)$ and a midpoint $p\in[xy]$ such that $d_{G_1\circ G_2}(p,[xz]\cup[zy])=\d(G_1)+3/2$. Let $V_p:=(v,w)$ be a closest vertex to $p$ in $[xy]\cap V(G_1\circ G_2)$ as in the proof of Theorem \ref{t:CotaSup}, \emph{i.e.}, $d_{G_1\circ \{w\}}(V_p, [\pi(x)\pi(z)]\cup[\pi(z)\pi(y)])=\d(G_1)$ with $\pi$ the canonical projection on $G_1\circ\{w\}$; besides, this equality is attained in a vertex $\a\in [\pi(x)\pi(z)]\cup[\pi(z)\pi(y)]$.
Note that $\d(G_1)$ is an integer number since it is the distance between two vertices.  Since $\d(G_1)>0$, we have $\d(G_1)\ge 1$.
Let $V_p'$ be the vertex in $T\cap V(G_1\circ G_2)$ such that $[V_p,V_p']$ is the edge in $G_1\circ G_2$ with $p\in[V_p,V_p']$.
Since $d_{G_1\circ G_2}(p,\{x,y\}) \ge d_{G_1\circ G_2}(p,[xz]\cup[zy])= \d(G_1)+3/2$, there exist $a,b\in [xy]\cap V(G_1\circ G_2)$ with $d_{G_1\circ G_2}(a,p)= d_{G_1\circ G_2}(b,p)= 3/2$ and $d_{G_1\circ G_2}(a,b)=3$. If $\pi(V_p)=\pi(V_p')$, then $d_{G_1\circ \{w\}}(\pi(a),\pi(b))=2$. This contradicts Lemma \ref{l:dist}, and so, we have $\pi(V_p)\neq \pi(V_p')$ and $\pi(V_p)\neq \pi(p)\neq \pi(V_p')$.
If $d_{G_1\circ\{w\}}(\pi(p), [\pi(x)\pi(z)]\cup[\pi(z)\pi(y)]) = d_{G_1\circ \{w\}}(V_p, [\pi(x)\pi(z)]\cup[\pi(z)\pi(y)]) = \d(G_1)\ge 1$, then since $\pi(V_p)\neq\pi(p)$ we obtain that $d_{G_1\circ\{w\}}(\xi, [\pi(x)\pi(z)]\cup[\pi(z)\pi(y)]) = \d(G_1)+1/4$ where $\xi$ is the midpoint of $[\pi(p) V_p]$. But this is a contradiction since $d_{G_1\circ\{w\}}(\xi, [\pi(x)\pi(z)]\cup[\pi(z)\pi(y)]) \le \d(G_1)$.
Then we have $d_{G_1\circ\{w\}}(\pi(p), [\pi(x)\pi(z)]\cup[\pi(z)\pi(y)]) < d_{G_1\circ \{w\}}(V_p, [\pi(x)\pi(z)]\cup[\pi(z)\pi(y)])=\d(G_1)$; hence,  $d_{G_1\circ\{w\}}(\pi(p), [\pi(x)\pi(z)]\cup[\pi(z)\pi(y)]) = \d(G_1) - 1/2$ and $d_{G_1\circ\{w\}}(\pi(V_p'), [\pi(x)\pi(z)]\cup[\pi(z)\pi(y)]) = \d(G_1) - 1$.
We can repeat the same argument in the proof of Theorem \ref{t:CotaSup} for $V_p'$ instead of $V_p$, and we obtain $d_{G_1\circ G_2}(p,[xz]\cup[zy]) \le \d(G_1)+1/2$.
This is the contradiction we were looking for and $G_1$ is a tree.

Hence, $\d(G_1\circ G_2)=3/2$. If $G_2$ is a trivial graph, then $3/2=\d(G_1\circ G_2)=\d(G_1)=0$, which is a contradiction. Therefore, $G_2$ is a non-trivial graph.
\end{proof}

Theorem \ref{th:hyp3/2_diam2} below is a converse of Theorem \ref{th:equal+3/2}; furthermore, it provides the exact value of the hyperbolicity constant of the lexicographic product of many trees and graphs. We need some lemmas.

\begin{lemma}\label{l:Tcopia}
Let $G_1$ be any tree with $1\le \diam G_1\le 2$ and $G_2$ any graph. Then $\d(G_1\circ G_2)=3/2$ if and only if there is a geodesic triangle $T = \{x, y, z\}$ in $G_1\circ G_2$ that is a cycle contained in $\{v_0\} \circ G_2$ for some $v_0\in V(G_1)$ with $x, y, z \in J(\{v_0\} \circ G_2)$ and a vertex $p \in [xy]$ such that $d _{G_1\circ G_2} (p, [xz] \cup [zy])= d _{G_1\circ G_2} (p, x)= d _{G_1\circ G_2} (p, y)= 3/2$.
\end{lemma}

\begin{proof}
Assume first that $\d(G_1\circ G_2)=3/2$.
By Theorem \ref{t:TrianVMp} there exists a geodesic triangle $T=\{x,y,z\}$ in $G_1 \circ G_2$ that is a cycle with $x,y,z \in J(G_1 \circ G_2)$ and a point $p\in [xy]$ such that $\d(T)=d_{G_1\circ G_2}(p,[yz]\cup [zx])=3/2$. Thus, $d _{G_1\circ G_2} (p, \{x,y\})\ge d _{G_1\circ G_2} (p, [xz] \cup [zy])= 3/2$ and $L([xy])\ge 3$.

Assume that $\diam G_1= 2$ (the case $\diam G_1=1$ is similar and simpler).
We show now that $\diam G_1\circ G_2=3$.
Note that $\diam G_1\circ G_2 \ge L([xy])\ge 3$.
Let $A,B\in V(G_1\circ G_2)$.
If $\pi(A)=\pi(B)$, then by Lemma \ref {l:dist} we have $d_{G_1\circ G_2}(A,B)\le 2$.
If $\pi(A)\neq \pi(B)$, then by Lemma \ref {l:dist} we have $d_{G_1\circ G_2}(A,B)\le 2$ since that $\diam G_1= 2$.
Therefore, $\diam V(G_1\circ G_2)=2$ and $\diam G_1\circ G_2\le 3$.
Consequently,  $\diam G_1\circ G_2=3$, $L([xy])=3$ and $d _{G_1\circ G_2} (p, x)= d _{G_1\circ G_2} (p, y)= 3/2$.
Notice that $x,y$ are midpoints of $G_1\circ G_2$ and $p$ a vertex of $G_1\circ G_2$.

Assume now that $x\in \{v_0\}\circ G_2$ for some $v_0\in V(G_1)$ and $y\notin \{v_0\}\circ G_2$, where $x\in [A_1,A_2]$ and $y\in [B_1,B_2]$ with $A_1,B_1\in [xy]$; then $d_{G_1\circ G_2}(A_1,B_1)=2$ since that $L([xy])=3$.
Note that $A_1\in \{v_0\}\times V(G_2)$ and $B_1\in \{w_0\}\times V(G_2)$ with $d_{G_1}(v_0,w_0)=2$.
We have that $[xy]\cap([yz]\cup [zx])=\{x,y\}$ since $T$ is a cycle.
Hence, $A_2,B_2\in V([yz]\cup [zx])$ and $d_{G_1\circ G_2}(p,[yz]\cup [zx])=d_{G_1\circ G_2}(p,\{A_2,B_2\})=1$ since $p$ is a vertex, and this is a contradiction.
If $y\in \{v_0\}\circ G_2$ for some $v_0\in V(G_1)$ and $x\notin \{v_0\}\circ G_2$, then the same argument gives a contradiction.
If $x,y \notin \cup_{v_0\in V(G_1)} \{v_0\}\circ G_2$, then one can check that $d_{G_1\circ G_2}(x,y)\le 2$, which is a contradiction.
Hence, we conclude that $x,y\in \{v_0\}\circ G_2$ for some $v_0\in V(G_1)$.
We also have $p\in \{v_0\}\circ G_2$ and we conclude that $[xy]$ is contained in $\{v_0\}\circ G_2$.
If $[yz]\cup [zx]$ is not contained in $\{v_0\}\circ G_2$, then there is a vertex $W\in [yz]\cup [zx]$ such that $W\in \{w_0\}\circ G_2$ and $d_{G_1}(v_0,w_0)=1$.
Hence, $d_{G_1\circ G_2}(p,W)=1$, which is a contradiction.
Then $T$ is contained in $\{v_0\}\circ G_2$.

It is easy to check that if there exists such a geodesic triangle $T$, then $\d(G_1\circ G_2)\ge \d(T)\ge3/2$. Theorem \ref{t:CotaSup} allows to conclude $\d(G_1\circ G_2)=3/2$.
\end{proof}

Now we define some families of graphs which will be useful.
Denote by $C_n$ the cycle graph with $n\ge3$ vertices and by $V(C_n):=\{v_1^{(n)},\ldots,v_n^{(n)}\}$ the set of their vertices such that $[v_n^{(n)},v_1^{(n)}]\in E(C_n)$ and $[v_i^{(n)},v_{i+1}^{(n)}]\in E(C_n)$ for $1\le i\le n-1$.
Let  $\mathcal{C}_6^{(1)}$ be the set of graphs obtained from $C_6$ by adding a (proper or not) subset of the set of edges $\{[v_2^{(6)},v_6^{(6)}]$, $[v_4^{(6)},v_6^{(6)}]\}$.
Let us define the set of graphs
$$\mathcal{F}_6:=\{\text{graphs containing, as induced subgraph, an isomorphic graph to some element of } \mathcal{C}_6^{(1)}\}.$$
Let  $\mathcal{C}_7^{(1)}$ be the set of graphs obtained from $C_7$ by adding a (proper or not) subset of the set of edges $\{[v_2^{(7)},v_6^{(7)}]$, $[v_2^{(7)},v_7^{(7)}]$, $[v_4^{(7)},v_6^{(7)}]$, $[v_4^{(7)},v_7^{(7)}]\}$.
Define
$$\mathcal{F}_7:=\{\text{graphs containing, as induced subgraph, an isomorphic graph to some element of } \mathcal{C}_7^{(1)}\}.$$
Let  $\mathcal{C}_8^{(1)}$ be the set of graphs obtained from $C_8$ by adding a (proper or not) subset of the set $\{[v_2^{(8)},v_6^{(8)}]$, $[v_2^{(8)},v_8^{(8)}]$, $[v_4^{(8)},v_6^{(8)}]$, $[v_4^{(8)},v_8^{(8)}]\}$.
Also, let $\mathcal{C}_8^{(2)}$ be the set of graphs obtained from $C_8$ by adding a (proper or not) subset of $\{[v_2^{(8)},v_8^{(8)}]$, $[v_4^{(8)},v_6^{(8)}]$, $[v_4^{(8)},v_7^{(8)}]$, $[v_4^{(8)},v_8^{(8)}]\}$.
Define
$$\mathcal{F}_8:=\{\text{graphs containing, as induced subgraph, an isomorphic graph to some element of } \mathcal{C}_8^{(1)}\cup \mathcal{C}_8^{(2)}\}.$$
Let  $\mathcal{C}_9^{(1)}$ be the set of graphs obtained from $C_9$ by adding a (proper or not) subset of the set of edges $\{[v_2^{(9)},v_6^{(9)}]$, $[v_2^{(9)},v_9^{(9)}]$, $[v_4^{(9)},v_6^{(9)}]$, $[v_4^{(9)},v_9^{(9)}]\}$.
Define
$$\mathcal{F}_9:=\{\text{graphs containing, as induced subgraph, an isomorphic graph to some element of } \mathcal{C}_9^{(1)}\}.$$
Finally, we define the set $\mathcal{F}$ by
$$\mathcal{F}:=\mathcal{F}_6\cup\mathcal{F}_7\cup\mathcal{F}_8\cup\mathcal{F}_9.$$
Note that $\mathcal{F}_6$, $\mathcal{F}_7$, $\mathcal{F}_8$ and $\mathcal{F}_9$ are not disjoint sets of graphs.

For any non-empty set $S\subset V(G)$, the induced subgraph of $S$ will be denoted by $\langle S\rangle$.

\begin{lemma}\label{l:hyp3/2_6-9}
Let $G$ be any graph. Then $G\in \mathcal{F}$ if and only if there is a geodesic triangle $T=\{x,y,z\}$ in $G$ that is a cycle with $x,y,z\in J(G)$, $L([xy]),L([yz]),L([zx])\le 3$ and $\d(T)=3/2=d_{G}(p,[yz]\cup[zx])$ for some $p\in[xy]$.
\end{lemma}

\begin{proof}
Assume first that there is a geodesic triangle $T=\{x,y,z\}$ in $G$ that is a cycle with $x,y,z\in J(G)$, $L([xy]),L([yz]),L([zx])\le 3$ and $\d(T)=3/2=d_{G}(p,[yz]\cup[zx])$ for some $p\in[xy]$.
Since $d_{G}(p,\{x,y\})\ge d_{G}(p, [yz]\cup[zx]) = 3/2$, we have $L([xy])=3$ and $p$ is the midpoint of $[xy]$, thus $p\in V(G)$.
Since $L([yz])\le3$, $L([zx])\le3$ and $L([yz])+L([zx])\ge L([xy])$, we have $6\le L(T)\le9$.

Assume now that $L(T)=6$.
Denote by $\{v_1,\ldots,v_6\}$ the vertices in $T$ such that $T=\bigcup_{i=1}^{6}[v_i,v_{i+1}]$ with $v_7:=v_1$. Without loss of generality we can assume that $x\in[v_1,v_2]$, $y\in[v_4,v_5]$ and $p=v_3$. Since $d_{G}(x,y)=3$, we have that $\langle\{v_1,\ldots,v_6\}\rangle$ contains neither $[v_1,v_4]$, $[v_1,v_5]$, $[v_2,v_4]$ nor $[v_2,v_5]$; besides, since $d_{G}(p,[yz]\cup[zx])>1$ we have that $\langle\{v_1,\ldots,v_6\}\rangle$ contains neither $[v_3,v_1]$, $[v_3,v_5]$ nor $[v_3,v_6]$.
Note that $[v_2,v_6]$, $[v_4,v_6]$ may be contained in $\langle\{v_1,\ldots,v_6\}\rangle$. Therefore, $G\in \mathcal{F}_6$.

Assume that $L(T)=7$ and $G\notin \mathcal{F}_6$.
Denote by $\{v_1,\ldots,v_7\}$ the vertices in $T$ such that $T=\bigcup_{i=1}^{7}[v_i,v_{i+1}]$ with $v_8:=v_1$. Without loss of generality we can assume that $x\in[v_1,v_2]$, $y\in[v_4,v_5]$ and $p=v_3$. Since $d_{G}(x,y)=3$, we have that $\langle\{v_1,\ldots,v_7\}\rangle$ contains neither $[v_1,v_4]$, $[v_1,v_5]$, $[v_2,v_4]$ nor $[v_2,v_5]$; besides, since $d_{G}(p,[yz]\cup[zx])>1$ we have that $\langle\{v_1,\ldots,v_7\}\rangle$ contains neither $[v_3,v_1]$, $[v_3,v_5]$, $[v_3,v_6]$ nor $[v_3,v_7]$.
Since $G\notin \mathcal{F}_6$, $[v_1,v_6]$ and $[v_5,v_7]$ are not contained in $\langle\{v_1,\ldots,v_7\}\rangle$.
Note that $[v_2,v_6]$, $[v_2,v_7]$, $[v_4,v_6]$, $[v_4,v_7]$ may be contained in $\langle\{v_1,\ldots,v_7\}\rangle$. Hence, $G\in \mathcal{F}_7$.

Assume that $L(T)=8$ and $G\notin \mathcal{F}_6\cup\mathcal{F}_7$.
Denote by $\{v_1,\ldots,v_8\}$ the vertices in $T$ such that $T=\bigcup_{i=1}^{8}[v_i,v_{i+1}]$ with $v_9:=v_1$. Without loss of generality we can assume that $x\in[v_1,v_2]$, $y\in[v_4,v_5]$ and $p=v_3$. Since $d_{G}(x,y)=3$, we have that $\langle\{v_1,\ldots,v_8\}\rangle$ contains neither $[v_1,v_4]$, $[v_1,v_5]$, $[v_2,v_4]$ nor $[v_2,v_5]$; besides, since $d_{G}(p,[yz]\cup[zx])>1$ we have that $\langle\{v_1,\ldots,v_8\}\rangle$ contains neither $[v_3,v_1]$, $[v_3,v_5]$, $[v_3,v_6]$, $[v_3,v_7]$ nor $[v_3,v_8]$.
Since $G\notin \mathcal{F}_6\cup\mathcal{F}_7$, $[v_1,v_6]$, $[v_1,v_7]$, $[v_5,v_7]$, $[v_5,v_8]$ and $[v_6,v_8]$ are not contained in $\langle\{v_1,\ldots,v_8\}\rangle$.
Since $T$ is a geodesic triangle we have that $z\in\{v_{6,7},v_7,v_{7,8}\}$ with $v_{6,7}$ and $v_{7,8}$ the midpoints of $[v_6,v_7]$ and $[v_7,v_8]$, respectively.
If $z=v_7$ then $\langle\{v_1,\ldots,v_8\}\rangle$ contains neither $[v_2,v_7]$ nor $[v_4,v_7]$.
Note that $[v_2,v_6]$, $[v_2,v_8]$, $[v_4,v_6]$, $[v_4,v_8]$ may be contained in $\langle\{v_1,\ldots,v_8\}\rangle$.
If $z=v_{6,7}$ then $\langle\{v_1,\ldots,v_8\}\rangle$ contains neither $[v_2,v_6]$ nor $[v_2,v_7]$.
Note that $[v_2,v_8]$, $[v_4,v_6]$, $[v_4,v_7]$, $[v_4,v_8]$ may be contained in $\langle\{v_1,\ldots,v_8\}\rangle$.
By symmetry, we obtain an equivalent result for $z=v_{7,8}$. Therefore, $G\in \mathcal{F}_8$.

Assume that $L(T)=9$ and $G\notin \mathcal{F}_6\cup\mathcal{F}_7\cup\mathcal{F}_8$.
Denote by $\{v_1,\ldots,v_9\}$ the vertices in $T$ such that $T=\bigcup_{i=1}^{9}[v_i,v_{i+1}]$ with $v_{10}:=v_1$. Without loss of generality we can assume that $x\in[v_1,v_2]$, $y\in[v_4,v_5]$ and $p=v_3$. Since $d_{G}(x,y)=3$, we have that $\langle\{v_1,\ldots,v_9\}\rangle$ contains neither $[v_1,v_4]$, $[v_1,v_5]$, $[v_2,v_4]$ nor $[v_2,v_5]$; besides, since $d_{G}(p,[yz]\cup[zx])>1$ we have that $\langle\{v_1,\ldots,v_9\}\rangle$ contains neither $[v_3,v_1]$, $[v_3,v_5]$, $[v_3,v_6]$, $[v_3,v_7]$, $[v_3,v_8]$ nor $[v_3,v_9]$.
Since $T$ is a geodesic triangle we have that $z$ is the midpoint of $[v_7,v_8]$.
Since $d_{G}(y,z)=d_{G}(z,x)=3$, we have that $\langle\{v_1,\ldots,v_9\}\rangle$ contains neither $[v_1,v_7]$, $[v_1,v_8]$, $[v_2,v_7]$, $[v_2,v_8]$, $[v_4,v_7]$, $[v_4,v_8]$, $[v_5,v_7]$ nor $[v_5,v_8]$.
Since $G\notin \mathcal{F}_6\cup\mathcal{F}_7\cup\mathcal{F}_8$, $[v_1,v_6]$, $[v_5,v_9]$, $[v_6,v_8]$, $[v_6,v_9]$ and $[v_7,v_9]$ are not contained in $\langle\{v_1,\ldots,v_9\}\rangle$.
Note that $[v_2,v_6]$, $[v_2,v_9]$, $[v_4,v_6]$, $[v_4,v_9]$ may be contained in $\langle\{v_1,\ldots,v_9\}\rangle$.
Hence, $G\in \mathcal{F}_9$.

Therefore, in any case $G\in \mathcal{F}$.

The previous argument also shows that if $G\in \mathcal{F}$, then there is a geodesic triangle with the required properties.
\end{proof}

Theorem \ref{th:equal+3/2} and the following result characterize the graphs for which the bound in Theorem \ref{t:CotaSup} is attained.

\begin{theorem}\label{th:hyp3/2_diam2}
Let $G_1$ be any tree and $G_2$ any non-trivial graph.
\begin{itemize}
\item [$(1)$]\quad If $\diam G_1\ge 3$, then $\d(G_1\circ G_2)=3/2$.

\item [$(2)$]\quad If $1\le \diam G_1\le 2$, then $\d(G_1\circ G_2)=3/2$ if and only if $G_2\in\mathcal{F}$.

\item [$(3)$]\quad If $G_1$ is trivial, then $\d(G_1\circ G_2)=3/2$ if and only if $\d(G_2)=3/2$.
\end{itemize}
\end{theorem}

\begin{proof}
If $\diam G_1\ge 3$, then Theorems \ref{t:Cota 5/4 G1} and \ref{t:CotaSup} give the result since that $\d(G_1)=0$.

In order to prove $(2)$, by Lemma \ref{l:Tcopia}, we have that $\d(G_1\circ G_2)=3/2$ if and only if there is a geodesic triangle $T = \{x, y, z\}$ in $G_1\circ G_2$ that is a cycle contained in $\{v\} \circ G_2$ for some $v\in V(G_1)$ with $x, y, z \in J(\{v\} \circ G_2)$ and a vertex $p \in [xy]$ such that $d _{G_1\circ G_2} (p, [xz] \cup [zy])= d _{G_1\circ G_2} (p, x)= d _{G_1\circ G_2} (p, y)= 3/2$.
By Lemma \ref{l:dist}, $\diam V(G_1\circ G_2)=2$, hence, $L([yz]),L([zx])\le3$ and $x,y$ are midpoints with $L([xy])=3$.
Hence, by Lemma \ref{l:hyp3/2_6-9} we have that $\d(G_1\circ G_2)=3/2$ if and only if $\{v\} \circ G_2\in \mathcal{F}$ and so, Remark \ref{r:remark2} gives that this is equivalent to $G_2\in \mathcal{F}$.

Finally, if $G_1$ is trivial, then Remark \ref{r:remark2} gives the result.
\end{proof}

The following result allows to compute, in a simple way, the hyperbolicity constant of the lexicographic product of any tree and any graph.

\begin{theorem}\label{t:tree_graph}
Let $G_1$ be any tree and $G_2$ any graph. Then
 \[\d(G_1\circ G_2)=\left\{
\begin{array}{ll}
\d(G_2), \quad &\mbox{if }\ G_1\simeq E_1,\\
0, \quad &\mbox{if }\ G_2\simeq E_1,\\
1,\quad &\mbox{if }\ \diam G_1=1 \quad  \mbox{and}\ \quad \, 1\le \diam G_2\le2,\\
5/4,\quad &\mbox{if }\ \diam G_1=1\quad  \mbox{and}\ \quad \diam G_2 > 2 \quad  \mbox{and}\ \quad G_2\notin \mathcal{F},\\
5/4,\quad &\mbox{if }\ \diam G_1=2 \quad  \mbox{and}\ \quad \diam G_2\ge 1\quad  \mbox{and}\ \quad G_2\notin \mathcal{F},\\
3/2,\quad &\mbox{if }\ 1\le \diam G_1\le 2 \quad \quad \, \,   \mbox{and}\ \quad G_2\in \mathcal{F},\\
3/2,\quad &\mbox{if }\ \diam G_1\ge 3 \quad  \mbox{and}\ \quad \diam G_2\ge 1.
\end{array}
\right.
\]
\end{theorem}

\begin{proof}
If $G_1\simeq E_1$ or $G_2\simeq E_1$, then we have the result by Remark \ref {r:remark2}.

If $\diam G_1=1$ and $1\le \diam G_2 \le 2$, then Theorems \ref{t:multk/4}, \ref{t:Cota 1}, \ref{t:CotaSup} and  \ref{th:hyp3/2_diam2} give $\d(G_1\circ G_2)\in \{1,5/4\}$ since  $G_2\notin \mathcal{F}$.
Seeking for a contradiction we can assume that $\d(G_1\circ G_2)=5/4$. Then by Theorem \ref{t:TrianVMp}  there is a geodesic triangle $T = \{x, y, z\}$ in $G_1\circ G_2$ that is a cycle with $x,y,z\in J(G_1\circ G_2)$  and a point $p\in [xy]$ such that $\d(T)=d_{G_1\circ G_2}(p,[yz]\cup [zx])=5/4$. Thus, $d _{G_1\circ G_2} (p, \{x,y\})\ge d _{G_1\circ G_2} (p, [xz] \cup [zy])= 5/4$, $L([xy])\ge 5/2$ and $x,y\in \{v\}\circ G_2$ for some $v\in V(G_1)$ since $\diam G_1=1$. This is a contradiction since $\diam G_2\le2$ and we conclude that $\d(G_1\circ G_2) = 1$.

If  $\diam G_1=1$ and $\diam G_2 > 2$ or $\diam G_1=2$ and $\diam G_2\ge 1$, then Theorems \ref{t:multk/4}, \ref {t:Cota 5/4 G1}, \ref {t:Cota 5/4 G2} and \ref {t:CotaSup} give $\d(G_1\circ G_2)\in \{5/4,3/2\}$.
Finally, since $G_2\notin \mathcal{F}$, Theorem \ref{th:hyp3/2_diam2} gives $\d(G_1\circ G_2)\neq 3/2$ and we have $\d(G_1\circ G_2)=5/4$.

If $1\le\diam G_1\le2$ and $G_2\in \mathcal{F}$ or $\diam G_1\ge 3$ and $\diam G_2\ge 1$, then we have the result by Theorem \ref {th:hyp3/2_diam2}.
\end{proof}

\begin{corollary}\label{c:trees}
Let $G_1,G_2$ be any trees. Then
 \[\d(G_1\circ G_2)=\left\{
\begin{array}{ll}
0, \quad &\mbox{if }\ G_1\simeq E_1 \quad \quad \quad \mbox{or}\ \quad \, \, \, G_2\simeq E_1,\\
1,\quad &\mbox{if }\ \diam G_1=1 \quad  \mbox{and}\ \quad 1\le \diam G_2\le 2,\\
5/4,\quad &\mbox{if }\ \diam G_1=1\quad  \mbox{and}\ \quad \diam G_2\ge 3 ,\\
5/4,\quad &\mbox{if }\ \diam G_1=2 \quad  \mbox{and}\ \quad \diam G_2\ge 1,\\
3/2,\quad &\mbox{if }\ \diam G_1\ge 3 \quad  \mbox{and}\ \quad \diam G_2\ge 1.
\end{array}
\right.
\]
\end{corollary}

\begin{corollary}\label{c:Pn_Pm}
Let $P_n, P_m$ be two path graphs. Then
 \[\d(P_n\circ P_m)=\left\{
\begin{array}{ll}
0, \quad &\mbox{if }\ n=1 \quad  \mbox{or}\ \quad \, \, \, \, m=1,\\
1,\quad &\mbox{if }\ n=2 \quad  \mbox{and}\ \quad m=2,3,\\
5/4,\quad &\mbox{if }\ n=2\quad  \mbox{and}\ \quad m\ge 4 \quad \mbox{or}\ \quad n=3 \quad  \mbox{and}\ \quad m\ge 2,\\
3/2,\quad &\mbox{if }\ n\ge 4 \quad  \mbox{and}\ \quad m\ge 2.
\end{array}
\right.
\]
\end{corollary}

\end{document}